\documentclass[a4paper]{amsart}

\usepackage{amsmath, amsthm, amssymb}
\usepackage{amsfonts}
\usepackage[normalem]{ulem} 
\usepackage{verbatim}
\usepackage{url}
\usepackage{hyperref}

\usepackage{longtable}
\usepackage{tabularx}
\usepackage[center]{caption}
\usepackage{xypic, mathtools}
\usepackage{mdwlist} 
\usepackage{diagbox}

\theoremstyle{plain}
\newtheorem{thm}{Theorem}
\newtheorem{cor}[thm]{Corollary}
\newtheorem{lemma}[thm]{Lemma}
\newtheorem{prop}[thm]{Proposition}

\theoremstyle{definition}
\newtheorem{defn}[thm]{Definition}
\newtheorem{example}[thm]{Example}

\theoremstyle{remark}
\newtheorem{rmk}[thm]{Remark}

\newcommand{\BC}{{\mathbb{C}}}

\newcommand{\BH}{{\mathbb{H}}}

\newcommand{\BQ}{{\mathbb{Q}}}

\newcommand{\BV}{{\mathbb{V}}}

\newcommand{\BZ}{{\mathbb{Z}}}

\newcommand{\CJ}{{\mathcal J}}

\newcommand{\ra}{{\ \longrightarrow\ }}

\DeclareFontFamily{OT1}{rsfs}{}
\DeclareFontShape{OT1}{rsfs}{n}{it}{<-> rsfs10}{}
\DeclareMathAlphabet{\curly}{OT1}{rsfs}{n}{it}

\begin{document}
\title{A Serre derivative for even weight Jacobi forms}
\author{Georg Oberdieck}
\maketitle

\begin{abstract}
 Using deformed or twisted Eisenstein Series, we construct a Jacobi-Serre derivative on even-weight Jacobi forms that generalizes the classical Serre derivative on modular forms. As an application, we obtain Ramanujan equations for the index $1$ Eisenstein series $E_{4,1}, E_{6,1}$ and a newly defined $E_{2,1}$. Finally, we relate the deformed Eisenstein Series directly to the classical first Jacobi theta function.
\end{abstract}

\setcounter{tocdepth}{1}
\tableofcontents
\addtocounter{section}{-1}

\section{Introduction}
\subsection{The Serre derivative}
Let $\tau \in \BH, q = e^{2 \pi i \tau}$ and let $F(\tau) \in M_k$ be a modular form of weight $k$. The differential $F' := \frac{1}{2 \pi i} \frac{ \partial F}{\partial \tau}$ of $F$ fails to be modular, but can be completed to a modular form by adding a multiple of the non-modular second Eisenstein series $E_2$. We obtain a differential operator on modular forms
\[ \partial^{S} : M_k \ra M_{k+2}, \quad F \mapsto F' - \frac{k}{12} E_2 F, \]
called the Serre derivative. By the finite dimensionality of the vector space $M_k$ of modular forms of given weight $k$, it is then easy to obtain differential equations among modular forms, e.g. the Ramanujan equations \cite{R} for the Eisenstein series $E_{4}$ and $E_6$,
\begin{equation} \partial^{S} E_{2} + \frac{1}{12} E_2^2 = - \frac{1}{12} E_4, \quad \partial^{S}(E_4) = - \frac{1}{3} E_6, \quad \text{and} \quad \partial^{S}(E_6) = - \frac{1}{2} E_4^2. \label{002} \end{equation}

\subsection{Deformed Eisenstein series}
Jacobi forms are a natural two-variable generalization of modular forms introduced by Eichler and Zagier in \cite{EZ}. Let $z \in \BC,p = e^{2 \pi i z}$ and let $F(z,\tau)$ be a Jacobi form of index $m$ and weight $k$. As in the modular case, the differentials $F' = \frac{1}{2 \pi i} \frac{ \partial F}{\partial \tau}$ and $F^{\bullet} = \frac{1}{2 \pi i} \frac{ \partial F}{\partial z}$ are no longer Jacobi forms. The topic of the paper considers the basic question how to complete these differentials (and also higher ones) to honest Jacobi forms. \\

Let $B_n$ be the Bernoulli numbers given by $x/(e^x - 1) = \sum_n B_n x^n/n!$; in particular we have $B_1 = -1/2$. Define the \emph{deformed} or \emph{twisted Eisenstein series} $J_n(z,\tau)$ for all $n \geq 0$ by
\begin{equation}
 J_n(z,\tau) = \delta_{n,1} \frac{p}{p - 1} + B_n - n \sum_{k, r \geq 1} r^{n - 1} (p^k + (-1)^n p^{-k}) q^{k r}. \label{000}
\end{equation}
The name of the $J_n$ reminds of the fact, that they restrict to the classical Eisenstein series $E_{2k}$ at $z = 0$,
\[  J_{2k}(0,\tau) = B_{2k} E_{2k}(\tau)\quad \quad \text{ and } \quad \quad J_{2k + 1}(0,\tau) = 0. \]
Under the elliptic and modular transformations, $J_n(z,\tau)$ transforms like a Jacobi form of index $0$ and weight $n$, but adds additional lower order terms. Using these functions, one can complete differentials of Jacobi forms and obtain differential operators on Jacobi forms. This was already observed by Gaberdiel and Keller in \cite{GK}. As a result, they find a series of differential operators on all Jacobi forms, starting with the classical Heat operator.

Here we use the same principle for the slightly weaker setting of differential operators that are defined only on even-weight Jacobi forms. As a new result, we give a natural and very intersting such operator of degree $2$.

Let $\CJ_{k,m}$ be the space of Jacobi forms of weight $k$ and index $m$.

\begin{thm} \label{MainThm}
For all $k, m \geq 0$, there is a differential operator, called the Jacobi-Serre derivative,
\[ \partial^{J} : \CJ_{2k,m} \ra \CJ_{2k + 2,m}, \]
such that for every $F(z,\tau) \in \CJ_{2k,m}$ we have
\begin{equation} (\partial^{J}F)(0, \tau) = \partial^{S}(F(0,\tau)). \label{001} \end{equation}
$\partial^{J}$ is given by the formula
\[ \partial^{J}(F) = F' - \frac{k}{12} E_2 F + \frac{1}{1 - 4m} \Big( F^{\bullet \bullet} - J_1 F^{\bullet} + m J_2 F - \frac{m}{6} E_2 F \Big). \]
\end{thm}

By equation \eqref{001}, $\partial^{J}$ directly generalizes the Serre derivative to Jacobi forms of even weight. \\

The main application of Theorem \ref{MainThm} and similar constructions for higher differential operators is to find differential equations for Jacobi forms. We examplify this application by stating the index 1 analogs of Ramanujan equation \eqref{002}.

Let 
\begin{equation} \phi(z,\tau) = \phi_{-2,1}(z,\tau) = \frac{\phi_{10,1}}{\Delta(\tau)} \label{005} \end{equation}
be one of the generators of the algebra of even-weight weak Jacobi forms (\cite{EZ}, Thm 9.3) and let
\[ \wp(z,\tau) = \frac{1}{(2\pi i)^2} \Big( \frac{1}{z^2} + \sum_{n \geq 1} (2n+1) 2 \zeta(2n+2) E_{2n+2} z^{2n} \Big) \]
be the Weierstrasse $\wp$ function. We define the analog of $E_2$ for Jacobi forms of index $1$.
\begin{defn} \label{006} $E_{2,1}(z,\tau) := \phi(z,\tau) \Big( E_2(\tau) \wp(z,\tau) - \frac{1}{12} E_4 \Big). $\end{defn}

Although $E_{2,1}$ has several particular properties reminding of $E_2(\tau)$, see Lemma \ref{204}, the definition is rather ad-hoc and it would be interesting to find a more conceptual approach to $E_{2,1}(z,\tau)$.\footnote{The function $E_{2,1}$ introduced by Choie in \cite{C} is different from ours.} We state the Ramanujan equation for index $1$ Jacobi forms.
\begin{cor} \label{MainCor} Let $E_{4,1}$ and $E_{6,1}$ be the Jacobi-Eisenstein series of index $1$ and weight $4$ and $6$ respectively. Then
 \begin{equation} \label{003}
  \begin{aligned}
  \partial^J E_{2,1} + \frac{1}{12} E_2 E_{2,1} + \frac{1}{16} E_4' \phi_{-2,1} & = - \frac{1}{12} E_{4,1} \\
  \partial^J E_{4,1} & = -\frac{1}{3} E_{6,1} \\
  \partial^J E_{6,1} & = -\frac{1}{2} E_4 E_{4,1}.
\end{aligned}
\end{equation}
\end{cor}
After restricting \eqref{003} to $z = 0$, we obtain Ramanujans original equations \eqref{002}.

\subsection{Theta functions} Unrelated to the differential operators above, we derive in the last part of the paper an interesting relation between deformed Eisenstein series and Jacobi theta functions. Let
\begin{equation} \theta_1(z,\tau) = -i q^{1/8} (p^{1/2} - p^{-1/2}) \prod_{m \geq 1} (1 - q^m) (1 - p q^m) (1 - p^{-1} q^m) \label{008} \end{equation}
be the classical first Jacobi theta function. A straight-forward computation shows, that the function $J_1(z,\tau)$ arises as the logarithmic derivative of $\theta_1$,
\begin{equation} \frac{\theta_1^\bullet}{\theta_1} = J_1. \label{004} \end{equation}
This can be generalized as follows. For a formal variable $x$, let
\begin{equation} \curly{J} = \sum_{n \geq 0} \frac{J_n}{n !} x^n \label{007} \end{equation}
be the generating functions for the $J_n$ functions.
\begin{thm} \label{ThetaThm} We have
\[ \curly{J} = x \cdot \theta_1^\bullet(0, \tau) \cdot \frac{ \exp(x \partial_z) \cdot \theta_1(z,\tau) }{ \theta_1(\frac{x}{2 \pi i},\tau) \theta_1(z,\tau) }, \]
where $\partial_z = \frac{1}{2 \pi i} \frac{\partial}{\partial z}$ and
\[ \exp(x \partial_z) \cdot \theta_1(z,\tau) := \sum_{k \geq 0} \frac{x^k}{k!} \partial_z^k( \theta_1(z,\tau) ). \]
\end{thm}

As an application, we obtain by a trivial relation among the deformed Eisenstein series a new\footnote{to the best of the author's knowledge} sequence of differential relations among the first theta function, see section \ref{thetaapplication}.

\subsection{Plan of the paper}
The first section concerns the study of the deformed Eisenstein series $J_n$. We first re-derive their transformation behaviour for modular ($(z,\tau) \mapsto (z/\tau, -1/\tau)$) and the elliptic ($(z,\tau) \mapsto (z + \lambda \tau + \mu, \tau)$) transformations. Then we complete $J_n$ to meromorphic Jacobi forms $K_n$ of index $0$ and weight $n$ via,
\[ K_n = \sum_{k = 0}^{n} (-1)^{n+k} \binom{n}{k} J_k J_1^{n-k}. \]
$K_n$ is an element of weight $n$ in the vectorspace $\BV_{n}$ of meromorphic Jacobi forms of index $0$ (i.e. double periodic ones) with only pole at $0$ of order $\leq n$. We show how to use this to easily derive relation among derivatives and products of the functions $J_n$. \\

In the second section, we prove the main theorem \ref{MainThm}. For this, we use the basic fact, that the vector space of Jacobi forms of index $m$, $\CJ_{\ast,m}$, is isomorphic to $\BV_m$ by the map $F \mapsto F/\phi^m$. After reexpressing differentials of Jacobi forms in $\BV_m$, we can easily write down differential operators for Jacobi forms. This gives a framework to also deal with more complicated differential equations and operators for Jacobi forms. We use this to find the definition for $\partial^J$ and prove Corollary \ref{MainCor}.\\

Finally, in the last section we prove Theorem \ref{ThetaThm} using the completions $K_n$. We also give a definition of deformed Eisenstein series $J_{i,n}$ corresponding to the other classical theta functions and prove an analogous statement for them.

\subsection{Relation to other work}
Deformed Eisenstein series were considered already in \cite{GK} and \cite{MTZ} in the process of studying $N = 2$ superconformal field theories and differential equations for elliptic genera (which are vector valued weak Jacobi Forms). In particular, in \cite{GK} Gaberdiel and Keller study the modular and periodic properties of deformed Eisenstein Series and the proofs given here are analog. By arguments from conformal field theory, \cite{GK} then obtain a set differential operators for (all) weak Jacobi forms. In contrast, our method is completely elementary. \\

Differential equation for Jacobi forms and deformed Eisenstein series appear also when studying Gromov-Witten invariants. The enumerative geometry of K3 surfaces and Hilbert schemes of K3 surfaces is encoded in various modular and Jacobi forms, see \cite{MPT}, \cite{KKV} and \cite{HilbK3}. In \cite{HilbK3} the calculation of the GW invariants are reduced to solving an explicit set of partial differential equations in 2 variables, that is obtained by applying WDVV equations (see \cite{FP}) in the case of the Hilbert scheme of $2$ points of $\mathbb{P}^1 \times E$. Here $E$ is a smooth elliptic curve. The solution to this system is given by Jacobi forms of index $1$ and deformed Eisenstein series. The equations give then complicated differential equations intertwining Jacobi forms and deformed Eisenstein series. Understanding this system was the author's main motivation for studying these functions in more generality.

\subsection{Acknowledgements} I would like to thank the following people. The programmers behind the math software \textit{SAGE} and \textit{mpmath} for their work. \"Ozlem Imamoglu, Jonas Jermann, Aaron Pixton, Martin Raum and Emanuel Scheidegger for various discussions and comments on the subject. And my advisor Rahul Pandharipande for his constant support and patience.

\section{Deformed Eisenstein series} \label{0000}
\subsection{Transformation properties}
We prove the transformation property of $J_n$ for the elliptic and modular transformations.
\begin{lemma} \label{102}
For $\lambda, \mu \in \BZ$,
\[ J_n(z + \lambda \tau + \mu, \tau) = \sum_{k = 0}^{n} (-1)^{n + k} \binom{n}{k} \lambda^{n-k} J_k(z,\tau). \]
\end{lemma}
\begin{proof}
Replace $p$ by $p q^{\lambda}$ in the right hand side of \eqref{000} and calculate.
\end{proof}

Note next, that we have for all $k \geq 1$ the basic relation
\begin{equation} \frac{k}{k+1} J_{k+1}^\bullet = J_{k}' \label{diffrelation}. \end{equation}
\begin{lemma}[\cite{GK}] $\displaystyle{ J_n(z / \tau, -1/ \tau) = \sum_{k = 0}^n \binom{n}{k} z^{n-k} \tau^k J_k }$. \end{lemma}
\begin{proof}
Using the Taylor expansion $p = \sum_{k \geq 0} (2 \pi i z)^k / k!$ in the Fourier expansion of $J_1$, one obtains
\begin{equation} J_1(z,\tau) = \frac{1}{w} + \sum_{n \geq 1} \frac{w^{2n-1}}{(2n)!} \cdot \Big( B_{2n} E_{2n}(\tau) \Big), \label{101} \end{equation}
where $w = 2 \pi i z$. By the transformation property of the Eisenstein series, we then have
\[ J_1(z / \tau, -1/ \tau) = z + \tau J_1(z,\tau). \]
By induction, we proceed now as follows. Suppose we know how $J_i(z,\tau)$ transforms under the substition $(z, \tau) \mapsto (z / \tau, - 1/ \tau)$. Then, by differentiating the transformation equation for $J_i$ with respect to $\tau$, and using \eqref{diffrelation}, we obtain an expression for $J_{i+1}^\bullet(z / \tau, -1/ \tau)$. Integrating with respect to $z$, we find an expression for $J_{i+1}(z/ \tau, -1/ \tau)$ up to a function that depends only on $\tau$. Plugging in $z = 0$ and using that $J_{2g}$ restricts to standard Eisenstein series, for which we know the transformation property, while $J_{2g + 1}$ restricts to $0$, we obtain the transformation law for $J_{i + 1}$.
\end{proof}

\subsection{The completion}
Define recursively functions $K_n(z,\tau)$ for $n \geq 2$ by
\begin{equation} K_n = J_n - J_1^n - \sum_{q = 2}^{n - 1} \binom{n}{q} K_q J_1^{n - q}, \label{Knrecursive} \end{equation}
where the sum is empty for $n = 2$. An explicit formula can be given by
\begin{equation} K_n = \sum_{k = 0}^{n} (-1)^{n+k} \binom{n}{k} J_k J_1^{n-k}. \label{100} \end{equation}

\begin{prop}
 $K_n$ are double-periodic in $z$ and are modular of weight $n$, that is for all $\lambda, \mu \in \BZ$
\begin{align*}
K_n(z + \lambda \tau + \mu, \tau) & = K_n(z, \tau) \\
K_n(z/\tau, -1/\tau) & = \tau^n K^n.
\end{align*}
\end{prop}
\begin{proof}
By induction on $n$ and a calculation using equation \eqref{Knrecursive}.
\end{proof}

\subsection{Poles}
Let
\[ D = \{ \lambda + \mu \tau\ |\ 0 \leq \lambda, \mu < 1 \} \]
be a fundamental region for $z$ with respect to a fixed $\tau$. By \eqref{004}, $J_1$ has a single pole of order $1$ at $0$ and no other pole in $D$. By \eqref{diffrelation}, $J_n$ then has for $n \geq 2$ no poles at all in $D$.\footnote{By Lemma \ref{102}, $J_n$ has poles outside of $D$.}
\begin{lemma} \label{poleKn}
For all $n \geq 2$, $K_n$ has a pole of order $n$ at $z = 0$ and no other poles in a fundamental region. Moreover, if $K_n = \sum a_k(\tau) w^k$, where $w = 2 \pi i z$, then $a_k$ are holomorphic modular forms of weight $k + n$, $a_{-1} = 0$ and $a_{-n} = (-1)^{n+1} (n-1)$. In particular,
\[ K_n(z,\tau) = \frac{(-1)^{n + 1} (n-1)}{w^n} + O(w^{-n+4}). \]
\end{lemma}
\begin{proof}
The first part is by \eqref{100} and the analysis of the poles of $J_n$. The statement on holomorphicity of $a_k(\tau)$ follows from the Fourier expansion of the $J_n$ for a fixed $z \neq 0$. $a_{-n}$ follows from the expansion \eqref{101} and $a_{-1} = 1$, since an elliptic function with a single pole has no residuum.
\end{proof}

Let $\BV$ be the $\BC$-vector space spanned by all meromorphic Jacobi forms\footnote{that is, a meromorphic function that satisfies the elliptic and modular transformation equation} $f : \BC \times \BH \ra \BC \cup \{ \infty \}$ of index $0$ and some weight, with only pole in the fundamental region at $0$ and a Laurent series at $0$ with coefficients holomorphic modular forms in $\tau$. By Lemma \ref{poleKn}, a basis of $\BV$ as a module over the ring of holomorphic modular forms $M_{\ast}$ is given by the $K_n$. We will use later the natural filtration
\[ \BV_n = \{ f \in \BV\ |\ f \text{ has a pole of order } \leq n \text{ at } 0 \}. \]

\subsection{Relations}
Two meromorphic Jacobi forms of index $0$ with the same principal part at their singularities are equal up to a function of $\tau$. Since for $n \leq 5$ we only have a single negative term in the Taylor expansion of $K_n$, we easily obtain relations among products and derivatives of the $K_i$ for low $i$. Moreover, using \eqref{100} and induction on $n$, we see that we can rewrite any derivative of $J_n$ in the form of products of $J_i$ for $i \leq n + 1$. We give the first few examples.

\begin{example}
\begin{align*}
K_2 & = J_1^{\bullet} - \frac{1}{12} E_2(\tau) = - \wp(z,\tau) \\
J_2^{\bullet} & = J_3 - J_1 J_2 + \frac{1}{6} E_2 J_1 \\
J_3^{\bullet} & = J_4 - J_3 J_1 + \frac{1}{4} J_2 E_2 - \frac{1}{120} E_4 \\
K_2 \cdot K_2 & = - \frac{1}{3} K_4 + \frac{1}{60} E_4,
\end{align*}
where we used \eqref{004} in the second equality of the first line.
\end{example}

\section{Differential operators}
\subsection{Reduction to $\BV$}
Let $\phi = \phi_{-2,1}$ be as in \eqref{005}. As $\phi(z,\tau) = \theta_1(z,\tau)^2 / \theta_1^{\bullet}(0,\tau)^2$ (see e.g. \cite{DMZ}), $\phi$ has a single zero at $0$ of order $2$ in the fundamental region. \\

Let $F$ be a (weak) Jacobi form of index $m$ and weight $k$ and consider $F/\phi^m$. $F/ \phi^m$ is a meromorphic Jacobi form of index $0$, weight $2k + m$ and has a single pole of order $\leq 2m$ at $0$ in the fundamental region. Moreover, since the coefficients of a Taylor expansion of $F$ are quasi-modular forms \cite{DMZ}, $F / \phi^m \in \BV_{2m}$. It is then easy to prove the following.

\begin{lemma} Let $\widetilde{\CJ}_{\ast, m}$ be the space of weak Jacobi forms of index $m$. The map
\begin{equation} \widetilde{\CJ}_{\ast,m} \ra \BV_{2m}, \quad F \mapsto F/\phi^m \label{200} \end{equation}
is an isomorphism.
\end{lemma}
We will use this lemma, to transform statements on (weak) Jacobi forms to $\BV_{2m}$.

\subsection{Operators on $\BV_{2m}$}
Define the three operators,
\begin{alignat*}{3}
\bullet & \text{ Multiplication by } K_i :\quad \quad  &    K_i \cdot : & \BV_n \ra \BV_{n+i},  \quad   & f & \mapsto K_i \cdot f \\
\bullet & \text{ Differentiation by } z:    &          D_z : & \BV_n \to \BV_{n + 1},   & f & \mapsto f^\bullet \\
\bullet & \text{ Differentiation by } \tau: &       D_\tau : & \BV_n \to \BV_{n + 2},   & f & \mapsto f' - J_1 f^{\bullet} - \frac{k}{12} E_2 f.
\end{alignat*}

For every operator $T$ of this form, we obtain via
\begin{equation} \widetilde{T} : F \mapsto \phi^m T( F/\phi^m ) \label{201} \end{equation}
an operator on meromorphic Jacobi forms of fixed index $m$. In general $\widetilde{T}$ will introduce poles to holomorphic Jacobi forms, namely $K_i, D_z, D_\tau$ give rise to to poles of order $i, 1, 2$ respectively. By using appropriate linear combinations of these operators, one can cancel the appearing poles and obtain operators defined on holomorphic Jacobi forms. We illustrate the method in degree $2$. \\

\emph{Case degree 2.} Consider the operators of degree $2$, $D_{\tau}, D_z^2$ and multiplication by $K_2$, obtained from the list above. The action on monomials $1/w^n$ and $1/w^{n-1}$ (with $w = 2 \pi i z$) is given by
\begin{alignat*}{2}
 D_\tau( \frac{1}{w^n} ) & = \frac{n}{w^{n + 2}} + O(w^{-n}) &
 D_\tau( \frac{1}{w^{n-1}} ) &= (n-1) \frac{1}{w^{n+1}} + O(w^{n-1}) \\
 D_z^2( \frac{1}{w^n} ) & = n (n + 1) \frac{1}{w^{n+2}} + O(w^{-n}) \quad & \quad
 D_z^2( \frac{1}{w^{n-1}} ) &= n (n-1) \frac{1}{w^{n+1}} + O(w^{n-1}) \\
 K_2 \cdot \frac{1}{w^n} & = \frac{-1}{w^{n + 2}} + O(w^{-n}) &
 K_2 \cdot \frac{1}{w^{n-1}} &= - \frac{1}{w^{n+1}} + O(w^{n-1}).
\end{alignat*}

One finds that
\[ D_H = 2n D_\tau - D_z^2 + n (n-1) K_2 \]
is the unique linear combination (up to scalar), that sends $\BV_n$ to $\BV_n$; by \eqref{201}, $D_H$ introduces then a differential operator
\begin{equation} D_H : \widetilde{J}_{\ast, m} \ra \widetilde{J}_{\ast + 2,m}. \label{202} \end{equation}
This is the classical Heat operator as found in \cite{EZ}, \cite{DMZ}, \cite{GK}. \\

Consider now the space $\widetilde{J}_{2 \ast, m}$ of even-weight weak Jacobi forms. Under \eqref{200}, $\widetilde{J}_{2 \ast,m}$ is isomorphic to the space $\BV_{2m}^{\text{even}}$ of even functions in $\BV_{2m}$. Therefore, to find an operator $\widetilde{J}_{2 \ast, m} \ra \widetilde{J}_{2 \ast, m}$ of degree $2$, we only need to consider the action of our 3 operators on the single monomial $1/(w^{2m})$, and not on $1/w^{2m-1}$. We obtain a second independent operator
\[ T_{\tau} = D_\tau + n K_2, \]
that, by \eqref{201} again, defines an operator on weak Jacobi forms,
\begin{equation} T_{\tau} : \widetilde{J}_{2 \ast, m} \ra \widetilde{J}_{2 \ast + 2, m}. \label{203} \end{equation}

It is known, that $D_H$ preserves not only weak, but also (full) Jacobi forms. We check the same for $T_{\tau}$.
\begin{prop} $T_{\tau}$ defines an operator $\CJ_{2k,m} \ra \CJ_{2k+2,m}$. \end{prop}
\begin{proof}
 
Let $F$ be an even weight Jacobi form of weight $2k$ and index $m$ and let $\widetilde{F} = F/\phi^m$. From before, we deduce that $T_\tau F$ is a holomorphic function and satisfies the elliptic and modular transformation equations. We need to show that $T_\tau F$ has a Fourier expansion of the form
\[ \sum_{n \geq 0} \sum_{\substack{r \in \BZ \\ r^2 \leq 4 n m}} c(n,r) p^r q^n. \]
Equivalently, see \cite{DMZ}, we need to show that $\forall \alpha, \beta \in \BQ$, 
\[ q^{m \alpha^2} T_\tau(F) (\alpha \tau + \beta, \tau) \]
is bounded for $\tau \rightarrow \infty$. We split this into two cases.

\textit{Case A.} Assume $\alpha \in \BQ \diagdown \BZ$ or $\beta \in \BQ \diagdown \BZ$. Then $\theta_1(\alpha \tau + \beta, \tau) \neq 0$ and $\widetilde{F}(\alpha \tau + \beta, \tau)$ is a well defined function of $\BH$. As $\theta_1(\alpha \tau + \beta, \tau)$ is a modular form, it vanishes to a fixed order at $\tau = \infty$ and hence so does $\widetilde{F}$. When applying $T_{\tau}$ to $\widetilde{F}$, we take derivatives with respect to $\tau$ and multiply with functions of the form $J_i$. The first does at most increase the order of convergence at $\tau = \infty$. To see that this is true also for the second, note two things: a) we may restrict to $0 \leq \alpha, \beta < 1$ (with $\alpha, \beta = 0$ is excluded) and b) by the Fourier expansion of $J_i$, $J_i(\alpha \tau + \beta,\tau)$ is bounded for $\tau \rightarrow \infty$. Therefore, $T_{\tau}(\widetilde{F})$ converges not worse then $\widetilde{F}$ for $\tau \mapsto \infty$. Applying $\phi^m$, the claim follows. \\

\textit{Case B.} Assume $\alpha, \beta \in \BZ$. Then we can equally well assume $\alpha = \beta = 0$ and we need to show that $T_q(F)(0, \tau)$ is bounded for $\tau \rightarrow \infty$.
Let $F = F_0 + w^2 F_2 + O(w^4)$, with $F_0, F_2$ quasi modular forms. Then
\[ (T_q F)(0, \tau) = F_0' + \Big( \frac{m}{6} - \frac{k}{12} \Big) E_2 F_0 - 2 F_2 \]
which is bounded for $\tau \rightarrow \infty$.
\end{proof}

\begin{proof}[Proof of Theorem \ref{MainThm}]
Define
\[ \partial^J = \frac{1}{1 - 4m} (T_{\tau} - D_H). \]
By the previous proposition, $\partial^J$ is an operator on Jacobi forms, $\CJ_{2k,m} \ra \CJ_{2k+2,m}$. The claims of the Theorems follow now from direct calculations.
\end{proof}

\begin{rmk}
The case of higher degree works completely analog; see \cite{GK} for a list of operators on all Jacobi forms. With the above method, one can find additional operators defined only on even or odd Jacobi forms.
\end{rmk}

\subsection{Ramanujan's equations}
Let $E_{2,1}(z,\tau)$ be defined as in Definition \ref{006}. The following is derived by straightforward means.
\begin{lemma} \label{204}
$E_{2,1}$ satisfies the following properties:
\begin{enumerate*}
 \item[(a)] holomorphic on $\BC \times \mathbb{H}$
 \item[(b)] has a Fourier expansion $E_{2,1}(z, \tau) = \sum_{n \geq 0} \sum_{\substack{r \in \BZ \\ r^2 \leq 4n}} c(n,r) p^r q^n $. In particular $c(n,r) = 0$ for $4n - r^2 < 0$.
 \item[(c)] satisfies the elliptic transformation equation, while the modular equation reads
\[ E_{2,1}(z/\tau, -1/\tau) = e^{\frac{2 \pi i z^2}{\tau}} \tau^2 E_{2,1} + \frac{1}{2 \pi i} e^{\frac{2 \pi i z^2}{\tau}} \tau \phi_{0,1} \]
 \item[(d)] $E_{2,1}(0, \tau) = E_2(\tau)$.
\end{enumerate*}
\end{lemma}

The first Fourier coefficients $c(n,r)$ of $E_{2,1}$ are given by
\begin{table}[ht]
\centering
\begin{tabular}{c | c c c c c c c c c}
\hline
\diagbox[width=1.0cm,height=0.7cm]{$n$}{$r$}  & $-4$ & $-3$ & $-2$ & $-1$ & $0$ & $1$ & $2$ & $3$ & $4$ \\[0.5ex]
\hline
$0$ & $0$ & $0$     & $0$      & $0$     & $1$      & $0$      & $0$    & $0$     & $0$ \\
$1$ & $0$ & $0$     & $1$      & $-28$   & $30$     & $-28$    & $1$    & $0$     & $0$ \\
$2$ & $0$ & $0$     & $30$     & $-264$  & $396$    & $-264$   & $30$   & $0$     & $0$ \\
$3$ & $0$ & $-28$   & $396$    & $-1620$ & $2408$   & $-1620$  & $396$  & $-28$   & $0$ \\
$4$ & $1$ & $-264$  & $2408$   & $-7944$ & $11430$  & $-7944$  & $2408$ & $-264$  & $1$
\end{tabular}
\end{table}

\begin{proof}[Proof of Corollary \ref{MainCor}] A direct check. \end{proof}

\section{The relation to theta functions}
\subsection{Proof of Theorem \ref{ThetaThm}}
Define functions $\widetilde{h}_n(\tau)$ by
\begin{equation} \frac{1}{\theta_1(z,\tau)} = \frac{1}{w} \sum_{n \geq 0} \widetilde{h}_n(\tau) w^{n}, \label{301} \end{equation}
where as before $w = 2 \pi i z$ and let
\[ h_n := n! \cdot \widetilde{h}_n(\tau) \theta_1^{\bullet}(0,\tau). \]
Here $0! = 1$ and $h_0 = 1$. As $\theta_1(z,\tau)$ is odd, $h_n = 0$ for all odd $n$.

For $n \geq 0$, set
\[ F_n(z, \tau) = \frac{1}{\theta_1} \Big( \sum_{k = 0}^n \binom{n}{k} h_{n-k} \theta_1^{k\bullet} \Big), \]
where we let $\theta_i^{k\bullet}$ (resp. $\theta_i^{k'}$) be the $k$'th derivative of $\theta_i$ with respect to $z$ (resp. $\tau$).

\begin{thm}\label{FnJn} $F_n = J_n$ for all $n \geq 0$. \end{thm}
Note that Theorem \ref{FnJn} directly implies Theorem \ref{ThetaThm}.

\begin{proof} 
Differentiationg the equation
\[ \theta_1(z + \lambda \tau,\tau) = - e^{-2 \pi i (\lambda z + \frac{1}{2} \lambda^2 \tau)} \theta_1(z, \tau) \]
we find
\[ \theta_1^{k \bullet}(z + \lambda \tau, \tau) = - \sum_{l = 0}^{k} (-1)^{l + k} \binom{k}{l} e^{-2 \pi i (\lambda z + \frac{1}{2} \lambda^2 \tau)} \lambda^{k - l} \theta_1^{l \bullet}. \]

Therefore, independent of $h_k$,
\begin{align*}
 F_n(z + \lambda \tau) & = \frac{1}{\theta_1} \left( \sum_{k = 0}^n \binom{n}{k} h_{n-k} \sum_{l = 0}^{k} (-1)^{l + k} \binom{k}{l} \lambda^{k - l} \theta_1^{l \bullet} \right) \\
 & = \frac{1}{\theta_1} \left( \sum_{k = 0}^n \sum_{l = 0}^k \binom{n}{n - k + l} \binom{n - k + l}{l} (-1)^{n - k} h_{k - l} \lambda^{n - k} \theta_1^{l \bullet} \right) \\
 & = \frac{1}{\theta_1} \left( \sum_{k = 0}^n (-1)^{n+k} \lambda^{n - k} \binom{n}{k} \sum_{l = 0}^k \binom{k}{l} h_{k-l} \theta_1^{l \bullet} \right) \\
 & = \sum_{k = 0}^n (-1)^{n+k} \binom{n}{k} \lambda^{n - k} F_k.
\end{align*}

We proceed by induction on $n$. For $n = 0$ nothing is to prove and $n = 1$ follows from \eqref{004}. Assume now, that the claim of the theorem is true for all $k < n$, with $n \geq 2$. Let
\[ \widetilde{K_n} = \sum_{k = 0}^{n} (-1)^{n+k} \binom{n}{k} F_k F_1^{n-k} \]
and note that the recursion relation \eqref{Knrecursive} holds for $\widetilde{K_n}$ as well.\\

Let $n = 2m$ be even. Then, for a fixed $\tau$, $F_{2m}(z,\tau)$ doesn't have any poles for $z \in \{ \lambda + \mu \tau \mid 0 \leq \lambda, \mu < 1 \}$. Indeed, $\theta_1^{2k \bullet} = 2^k \theta_1^{k'}$ has a zero of order 1 at $z = 0$ and hence $\theta_1^{2k \bullet}/\theta_1$ extends to a holomorphic function at $z = 0$. By induction we conclude that the principal part of $\widetilde{K_n}$ equals the principal part of $K_n$. Therefore it is left to show that $F_{2m}(0) = J_{2m}(0) = B_{2m} E_{2m}$. This is equivalent to the identity,
\begin{equation} \sum_{k = 0}^{m} \binom{2m}{2k} h_{2m-2k} \frac{\theta_1^{2k\bullet}}{\theta_1}(0) = B_{2m} E_{2m}. \label{300} \end{equation}

As $\theta_1^{2k\bullet}/\theta_1$ is an even holomorphic function,
\[ \Big( \frac{\theta_1^{2k\bullet}}{\theta_1} \Big)^{\bullet} = \frac{\theta_1^{(2k+1)\bullet}}{\theta_1} - \frac{\theta_1^{2k\bullet}}{\theta_1} \frac{\theta_1^{\bullet}}{\theta_1} \]
vanishes to first order at $0$. Comparing poles and using $\theta_1^{\bullet}/\theta_1 = 1/w + O(w)$, we obtain
\begin{equation} (\theta_1^{2k \bullet}/\theta_1)(0,\tau) = \text{Res}_{w = 0}\Big( \frac{\theta_1^{(2k+1)\bullet}}{\theta_1} \Big) = \frac{\theta_1^{(2k+1)\bullet}(0,\tau)}{\theta_1^\bullet(0,\tau)}, \label{302} \end{equation}
where we used the Taylor expansion $\theta_1 = \sum_{k \geq 0} \frac{\theta_1^{(2k+1)\bullet}(0)}{(2k + 1)!} w^{2k+1}$ for the second equation and $\text{Res}$ denotes the residuum.

Therefore \eqref{300} reduces to
\[ \sum_{k = 0}^{m} \binom{2m}{2k} (2m - 2k)! \widetilde{h}_{2(m-k)} \theta_1^{(2k+1)\bullet}(0) = B_{2m} E_{2m}, \]
which follows from comparing the $2m-1$-th Taylor coefficient of the left and right hand side of $\frac{1}{\theta_1(z,\tau)} \cdot \theta_1(z,\tau)^\bullet = J_1$.

The case $n = 2m + 1$ odd is similar and ommited.
\end{proof}

\begin{rmk}
Comparing the $w^{2n}$ coefficient of $1/\theta_1(z,\tau) \cdot \theta_1(z,\tau) = 1$ using \eqref{301}, we obtain the relation
\begin{equation} \sum_{k = 0}^{m} \widetilde{h}_{2m-2k}(\tau) \frac{\theta_1^{(2k+1)\bullet}(0,\tau)}{(2k+1)!} = 0. \label{303} \end{equation}
By \eqref{302},
\[ \theta_1^{(2k+1)\bullet}(0,\tau) = \theta_1^{\bullet}(0,\tau) \cdot (\theta_1^{2k \bullet}/\theta_1)(0,\tau) = \theta_1^{\bullet}(0,\tau) \cdot (\theta_1^{k \prime}/\theta_1)(0,\tau). \]
Let $P_k = (\theta_1^{k \prime}/\theta_1)(0,\tau)$. Then using \eqref{008}, $P_1 = E_2/8$ and taking the derivative of $P_n$,
\begin{equation} P_{n+1} = P_n' + \frac{1}{8} E_2 P_n. \label{304} \end{equation}
By \eqref{303} and \eqref{304}, we obtain a recursion relation for the function $h_n$.

The first few non-trivial identities given by Theorem \ref{FnJn} then read,
\begin{align*}
J_2 & = \frac{1}{\theta_1} ( \theta_1^{\bullet \bullet} - \frac{1}{12} E_2 \theta_1 ) \\
 J_3 & = \frac{1}{\theta_1} ( \theta_1^{\bullet \bullet \bullet} - \frac{1}{4} E_2 \theta_1^\bullet ) \\
J_4 & = \frac{1}{\theta_1} ( \theta_1^{\bullet \bullet \bullet \bullet} - \frac{1}{2} E_2 \theta_1^{\bullet \bullet} + (-\frac{1}{10} E_2' + \frac{7}{240} E_2^2) \theta_1.
\end{align*}
\end{rmk}

\subsection{Applications} \label{thetaapplication}
Let $\curly{J} = \sum_{n \geq 0} J_n x^n/n!$ as in \eqref{007}. By \eqref{diffrelation},
\[ \partial_{\tau} \curly{J} = \Big( \frac{\partial}{\partial x} - \frac{1}{x} \Big) \partial_z \curly{J}, \]
where $\partial_{\tau} = \frac{1}{2 \pi i} \frac{\partial}{\partial \tau}$ and $\partial_{z} = \frac{1}{2 \pi i} \frac{\partial}{\partial z}$. By Theorem \ref{ThetaThm}, this implies relation among differentials of $\theta_1(z,\tau)$. For example extracting the $x^3$ coefficient, we have
\begin{cor}
\[ \theta_1^{\bullet\bullet\bullet \bullet} \theta_1 - 4 \theta_1^{\bullet \bullet \bullet} \theta_1^\bullet + 3 \theta_1^{\bullet \bullet} \theta_1^{\bullet \bullet} - \theta_1 \theta_1^{\bullet \bullet} E_2 + (\theta_1^{\bullet})^2 E_2 + \frac{1}{2} \theta_1^2 E_2' = 0 \]
\end{cor}

\subsection{The other theta functions}
We consider an analog of deformed Eisenstein series corresponding to the theta functions $\theta_2, \theta_3, \theta_4$. For $n \geq 1$, define
\begin{align*}
 J_{2,n} & = 2 J_n(2z, 2\tau) - J_n(z, \tau) \\
 J_{3,n} & = 2^{2-n} J_n(2z, \tau) - 2 J_n(2z,2\tau) + J_n(z, \tau) - 2^{1 - n} J_n(z, \tau/2) \\
 J_{4,n} & = \frac{1}{2^{n - 1}} J_n(z, \tau/2) - J_n(z, \tau).
\end{align*}
Concretely, we have
\begin{equation} \label{305}
\begin{aligned}
 J_{2,n}(z, \tau) & = \delta_{n,1} \frac{p}{p + 1} + B_n - n \sum_{k, r \geq 1} (-1)^k r^{n - 1} (p^k + p^{-k}) q^{k r} \\
 J_{3,n}(z, \tau) & = -B_n(1 - \frac{1}{2^{n - 1}}) - n \sum_{k , r \geq 1} (r - \frac{1}{2})^{n - 1} (-1)^k (p^k + (-1)^n p^{-k}) q^{k (r - \frac{1}{2})} \\
 J_{4,n}(z, \tau) & = -B_n(1 - \frac{1}{2^{n - 1}}) - n \sum_{k , r \geq 1} (r - \frac{1}{2})^{n - 1} (p^k + (-1)^n p^{-k}) q^{k (r - \frac{1}{2})}.
\end{aligned}
\end{equation}

The statements of section \ref{0000} apply with minor modifications also to the $J_{i,n}$. In particular we can define periodic $K_{i,n}$, find relations and express the derivatives of $J_{i,n}$ in terms of $J_{i,n}$ itself.

Let
\begin{align*}
 \theta_2(z,\tau) & = \theta_1(z + \frac{1}{2}, \tau) \\
 \theta_3(z,\tau) & = q^{1/8} p^{1/2} \theta_1(z + \frac{1}{2} \tau + \frac{1}{2}, \tau) \\
 \theta_4(z,\tau) & = \theta_3(z + 1/2, \tau) = -i q^{1/8} p^{1/2} \theta_1(z + 1/2 \tau, \tau)
\end{align*}
be the other theta functions. We state the analog of Theorem \ref{ThetaThm}.
\begin{thm} We have
\[ \sum_{n \geq 0} J_{i,n}(z,\tau) \frac{x^n}{n!} = x\ \frac{\theta_1^\bullet(0, \tau)}{\theta_1(\frac{x}{2 \pi i})}\ \frac{ \exp(x \partial_p) \cdot \theta_i(z,\tau) }{\theta_i(z,\tau) }. \]
\end{thm}
\begin{proof}
With \eqref{305} one proves the formulas
\begin{align*}
 J_{2,n}(z, \tau) & = J_{n}(z + \frac{1}{2}) \\
 J_{3,n}(z, \tau) & = \sum_{l = 0}^n \binom{n}{l} \frac{1}{2^{n - l}} J_l(z + \frac{1}{2} + \frac{1}{2} \tau) \\
 J_{4,n}(z, \tau) & = \sum_{l = 0}^n \binom{n}{l} \frac{1}{2^{n - l}} J_l(z + \frac{1}{2} \tau).
\end{align*}
The claims then reduces directly to Theorem \ref{ThetaThm}.
\end{proof}

\bibliographystyle{hep}
\bibliography{SD}

\vspace{+8 pt}
\noindent
Departement Mathematik\\
ETH Z\"urich\\
georgo@math.ethz.ch

\end{document}